\pdfoutput=1

\documentclass[a4paper,reqno,final]{amsart}

\usepackage[utf8]{inputenc}
\usepackage[T1]{fontenc}
\usepackage{amsmath, amssymb, amsthm}
\usepackage{dsfont}
\usepackage{tikz}
\usepackage{mathtools}
\usepackage[english]{babel}

\usepackage[babel]{csquotes}
\usepackage[backend=biber,style=alphabetic]{biblatex}
\usepackage[protrusion=true,expansion=true,babel=true,final]{microtype}
\usepackage[unicode,bookmarksopen]{hyperref}

\numberwithin{equation}{section}

\newcommand{\T}{\textrm T}

\newcommand{\Tr}{\mathit{Tr}}
\newcommand{\MRs}{\mathit{MR}}

\newcommand{\1}{\mathds{1}}
\newcommand{\K}{\mathds{K}}
\newcommand{\R}{{\mathds{R}}}
\renewcommand{\C}{\mathds{C}}
\newcommand{\N}{\mathds{N}}
\newcommand{\A}{\mathcal{A}}
\newcommand{\B}{\mathcal{B}}

\newcommand{\D}{\mathcal{D}}
\newcommand{\V}{\mathcal{V}}
\newcommand{\W}{\mathcal{W}}
\renewcommand{\L}{\mathcal{L}}
\renewcommand{\d}{\mathrm{d}}
\renewcommand{\Re}{\operatorname{Re}}

\newcommand{\fra}{\mathfrak{a}}
\newcommand{\frb}{\mathfrak{b}}
\newcommand{\frc}{\mathfrak{c}}

\newcommand\norm[1]{\lVert #1 \rVert}
\newcommand\abs[1]{\lvert #1 \rvert}
\newcommand{\normalabs}[1]{\lvert#1\rvert}

\makeatletter
\newcommand{\LeftEqNo}{\let\veqno\@@leqno}
\makeatother

\theoremstyle{plain}
\newtheorem{theorem}{Theorem}[section]
\newtheorem{proposition}[theorem]{Proposition}

\theoremstyle{remark}
\newtheorem{remark}[theorem]{Remark}

\theoremstyle{definition}
\newtheorem{example}[theorem]{Example}
\newtheorem{definition}[theorem]{Definition}
\newtheorem{problem}[theorem]{Problem}

\DeclareSourcemap{
  \maps[datatype=bibtex, overwrite]{
    \map{
      \step[fieldset=abstract, null]
      \step[fieldsource=entrykey,match=Mer99,final] %
      \step[fieldset=shorthand,fieldvalue=LeM99] %
    }
  }
}

\renewbibmacro{publisher+location+date}{%
	\printlist{publisher}
	\setunit*{\addcomma\space}
	\printlist{location}
  	\setunit*{\addcomma\space}
  	\usebibmacro{date}
\newunit} 

\newbibmacro{string+doi+url}[1]{%
	\iffieldundef{doi}{%
			\iffieldundef{url}{#1}{\href{\thefield{url}}{#1}}%
		}%
		{\href{http://dx.doi.org/\thefield{doi}}{#1}}
}%

\renewbibmacro{in:}{}

\DeclareFieldFormat*{title}{\usebibmacro{string+doi+url}{\mkbibemph{#1}}}
\DeclareFieldFormat*{booktitle}{#1}
\DeclareFieldFormat[article]{volume}{\textbf{#1}\addcomma}
\DeclareFieldFormat[article]{number}{\addspace no.~#1}
\DeclareFieldFormat[article]{pages}{#1}
\DeclareFieldFormat{journaltitle}{#1} 
\DeclareFieldFormat{url}{}
\DeclareFieldFormat{eprint}{Preprint on arXiv: \href{http://arxiv.org/abs/#1}{#1}} 

\renewcommand*{\bibfont}{\footnotesize}

\allowdisplaybreaks
\title{J.~L.~Lions' Problem on Maximal Regularity}
\author{Wolfgang Arendt}
\author{Dominik Dier}
\author{Stephan Fackler}
\address{Institute of Applied Analysis, Ulm University, Helmholtzstr.\ 18, 89069 Ulm}
\email{wolfgang.arendt@uni-ulm.de, dominik.dier@uni-ulm.de, \newline stephan.fackler@uni-ulm.de}
\thanks{The first and third author were supported by the DFG grant AR 134/4-1 ``Regularität evolutionärer Probleme mittels Harmonischer Analyse und Operatortheorie''.}

\begin{document}

\begin{abstract}
	This is a survey on recent progress concerning maximal regularity of non-autonomous equations governed by time-dependent forms on a Hilbert space. It also contains two new results showing the limits of the theory.
\end{abstract}

\maketitle

\section{Introduction}

The purpose of this survey is to describe the history and the
state of the art of J.~L.~Lions' problem on maximal regularity
for non-autonomous forms.
In particular, we formulate the problem in the concrete case 
of parabolic equations for which it is open.
We explain some consequences of a positive answer for quasi-linear parabolic 
equations. But we also present new results. The counterexample in Section~\ref{sec:counterexample} is published here for the first time. Section~\ref{sec:critical} on the critical case is new.

\section{Autonomous forms}\label{sec:autonomous_forms}

Throughout this note $H$ and $V$ are Hilbert spaces over $\K= \R$ or $\C$
such that $V$ is continuously embedded into $H$ and also dense in $H$.
We identify $h \in H$ with the functional $(h|\cdot)_H$ in $V'$ and thus we obtain
the Gelfand triple $V \hookrightarrow H \hookrightarrow V'$.
The spaces $V$ and $H$ are fixed and will not be mentioned explicitly further on.
An \emph{autonomous form} is a continuous, sesquilinear mapping $\fra \colon V \times V \to \K$.
Assume that the form is \emph{coercive}; i.e.,
\[
	\Re \fra(v,v) \ge \alpha\norm{v}^2_V \quad \text{for all }v \in V
\]
and some $\alpha >0$.
Then we associate the operator $\A \in \L(V,V')$ with $\fra$ by defining $\A v = \fra(v, \cdot)$ for $v \in V$.
Then $-\A$ generates a holomorphic semigroup on $V'$.
Frequently, the part $A$ of $\A$ in $H$ given by
\[
	D(A) = \{ v \in V : \A v \in H \},\quad Av = \A v
\]
is more important.
We call $A$ \emph{the operator on $H$ associated with $\fra$}.
This operator is suitable to incorporate diverse boundary conditions. Also $-A$ generates a contractive holomorphic
$C_0$-semigroup on $H$. We mention en passant that precisely those operators on $H$ which have a bounded $H^{\infty}$-calculus come from a form in that way~\cite{AreBuHaa01}. For the definition of fractional powers used below we refer to~\cite{Haa06}.

\begin{definition}
	The form $\fra$ has the \emph{Kato square root property} if $V=D(A^{1/2})$.
\end{definition}

McIntosh \cite{McI72} gave an example of a form that does not have the square root property.
By \cite[Theorem~1]{Kat62} it follows that there exists an example for which
\begin{equation}\label{eq:Kato}
	V \not\subset D(A^{1/2}) \quad\text{and}\quad  D(A^{1/2}) \not \subset V.
\end{equation}
Observe that we may take a direct sum to violate both inclusions.
However, if $\fra$ is \emph{symmetric}; i.e., $\fra(v,w) = \overline{\fra(w,v)}$ for all $v,w \in V$,
then the square root property is fulfilled: since $A^{1/2} = A^{1/2 *}$ and
\begin{equation*}
	\alpha \norm{v}_V^2 \le \fra(v,v) = (A^{1/2}v | A^{1/2*}v)_H = \norm{A^{1/2} v}_H^2 \le M \norm{v}_V^2 \qquad \text{for all } v \in D(A).
\end{equation*}
We give an example to illustrate how Neumann boundary conditions
are incorporated into the operator $A$.
Many further examples, e.g.\ Dirichlet and Robin boundary conditions, are well-known
and of importance.
The choice of dimension 1 is just for simplicity.

\begin{example}[the Neumann Laplacian]\label{ex:neumann}
	Let $H= L^2(0,1)$, $V=H^1(0,1)$, $\K = \R$, 
	$m \colon (0,1) \to [\delta, \frac 1 \delta]$ measurable,
	where $0<\delta<1$.
	Define the coercive form $\fra \colon V \times  V \to \R$ by
	\[
		\fra(v,w) = \int_0^1 m v' w' \ \d{x} + \int_0^1 v w \ \d{x}.
	\]
	No boundary condition is visible if we consider the operator $\A \colon H^1(0,1) \to (H^1(0,1))'$.
	However, its part $A$ in $L^2(0,1)$ is given by
	\begin{align*}
		D(A)&= \{ v \in H^1(0,1) : m v' \in H^1(0,1), (mv')(0) = (mv')(1)=0 \}\\
		Av&= -(m v')'.
	\end{align*}
	Recall that $H^1(0,1) \subset C([0,1])$.
	Choosing the unique continuous representative, the Neumann boundary condition incorporated into $D(A)$ makes sense.
\end{example}

\section{Non-autonomous forms}

Let $T>0$ and let $\fra \colon [0,T] \times V \times V \to \K$ be a \emph{non-autonomous form}; i.e., $\fra( \cdot,v,w) \colon [0,T]\to\K$ is measurable for all $v,w \in V$ and
\begin{equation*}
	\abs{\fra(t,v,w)} \le M \norm{v}_V \norm{w}_V \quad \text{for all } t\in [0,T],\, v,w \in V
\end{equation*}
and some $M\ge0$. Further, we assume that $\fra$ is \emph{coercive}; i.e.,
\[
	\Re \fra(t,v,v) \ge \alpha \norm{v}_V^2 \quad \text{for all } t\in [0,T],\, v \in V
\]
and some $\alpha >0$.
As before we consider $\A(t) \in \L(V,V')$ given by $\A(t)v = \fra(t,v, \cdot)$. If $X, Y$ are Hilbert spaces such that $X \hookrightarrow Y$ (i.e.\ $X$ is continuously embedded in $Y$),
then we define the Hilbert space 
\begin{equation*}
	\MRs(X,Y) \coloneqq \{ u \in C([0,T];Y): u \in L^2(0,T;X) \cap H^1(0,T;Y) \}.
\end{equation*}
In particular, since we consider throughout $V \hookrightarrow V'$ with $H$ as pivot, we have $\MRs(V,V') = L^2(0,T;V) \cap H^1(0,T;V')$.
This is the \emph{maximal regularity space} of the solutions in Lions' theorem below.
Note that $\MRs(V,V') \hookrightarrow C([0,T];H)$.
Using this notation, we can formulate Lions' well-posedness result with maximal regularity in $V'$.

\begin{theorem}\label{thm:Lions}
	Let $u_0 \in H$ and $f\in L^2(0,T;V')$.
	Then there exists a unique $u \in \MRs(V,V')$ such that
	\begin{equation*}
		\LeftEqNo
		\label{eq:ncp}
		\tag{NCP}
		\left\{ \begin{aligned}
		&u'(t)+ \A(t)u(t) = f(t) \quad t\text{-a.e.}\\
		&u(0)= u_0.
		\end{aligned} \right.
	\end{equation*}
\end{theorem}

Note that the terms $u'$, $\A(\cdot)u(\cdot)$ and $f$ lie in the space $L^2(0,T;V')$, which is the reason for the terminology ``maximal regularity''.
However, as we saw before, the right operator for solving a concrete problem is the part $A(t)$ of $\A(t)$ in $H$.
So the central problem can be formulated as follows.

\begin{problem}[Lions' Problem]
Let $f \in L^2(0,T;H)$. Under which conditions on the form $\fra$ the solution $u \in \MRs(V,V')$ of~\eqref{eq:ncp} satisfies $u \in H^1(0,T;H)$?
\end{problem}

Lions asked this question for several conditions on the form and on the initial value. He also gave partial positive answers 
as we will explain below.
It is illuminating to treat the problem of maximal regularity in $H$ for the initial value $u_0 = 0$ first
and to deal with other initial conditions by identifying the trace space later on. We start to define what we desire for $u_0 = 0$.

\begin{definition}
	A non-autonomous coercive form $\fra \colon [0,T] \times V \times V \to \K$
	satisfies \emph{maximal regularity in $H$} if for $u_0 = 0$ and each $f \in L^2(0,T;H)$ the solution $u \in \MRs(V,V')$ of~\eqref{eq:ncp} is in $H^1(0,T;H)$.
\end{definition}

As a consequence, $u(t) \in D(A(t))$ a.e.\ and $u'(t)+A(t)u(t)=f(t)$ for almost every $t \in [0,T]$.
Thus all three functions $u', A(\cdot)u(\cdot)$ and $f$ are in $L^2(0,T;H)$, which is the reason for the terminology ``maximal regularity in $H$''.
As a consequence, the solution is in the \emph{maximal regularity space with respect to $\fra$ and $H$}, namely
\[
	\MRs_\fra(H) := \{ u \in \MRs(V,H) : \A(\cdot) u(\cdot) \in L^2(0,T;H) \}.
\]
This is a Hilbert space for the norm
\[
	\norm{u}^2_{\MRs_\fra(H)} = \norm{u'}_{L^2(0,T;H)}^2 + \norm{\A(\cdot) u(\cdot)}^2_{L^2(0,T;H)}.
\]
We define the corresponding \emph{trace space} by $\Tr(\fra) := \{ u(0) : u \in \MRs_\fra(H) \}$
which is a Banach space for the norm
\[
	\norm{x}_{\Tr(\fra)} := \inf\{ \norm{u}_{\MRs_\fra(H)} : u\in \MRs_\fra(H), u(0)=x \}.
\]
If $u \in \MRs_\fra(H)$ is a solution of~\eqref{eq:ncp}, it follows that $u_0 \in \Tr(\fra)$.
Conversely, if $\fra$ satisfies maximal regularity in $H$, then for each $u_0 \in \Tr(\fra)$ there exists a unique solution $u \in \MRs_\fra(H)$ of \eqref{eq:ncp}. In fact, let $u_0 \in \Tr(\fra)$ and $f \in L^2(0,T;H)$.
There exists $v \in \MRs_\fra(H)$ such that $v(0)=u_0$.
Then $g= v'+ \A(\cdot) v(\cdot) \in L^2(0,T;H)$.
By assumption, there exists $w \in \MRs_\fra(H)$ such that $w'+\A(\cdot)w(\cdot) = f-g$ and $w(0)=0$.
Thus $u:=v+w \in \MRs_\fra(H)$ solves~\eqref{eq:ncp}.

Consequently, there are two tasks:
Finding conditions on the form $\fra$ that imply maximal regularity in $H$, and then
identifying the trace space $\Tr(\fra)$. For concrete problems, the given space $V$ is known and a desirable situation occurs when $\Tr(\fra) = V$. One even would like to have $\MRs_\fra(H) \subset C([0,T];V)$.
However, these properties are not valid in general as we will see in the subsequent sections,
where diverse regularity conditions on the form will be presented.
We start with the autonomous case where we already encounter a major difficulty for identifying the trace space.
\section{Autonomous forms: Regularity}\label{sec:autonomous_forms:regularity}

Let $\fra \colon V \times V \to \K$ be an autonomous, coercive form, $\A \in \L(V,V')$ the associated operator and $A$ the part of $\A$ in $H$.
Then the form $\fra$ has maximal regularity in $H$ and
$\MRs_\fra(H) = H^1(0,T;H) \cap L^2(0,T;D(A))$.
It follows from the trace method for real interpolation that $\Tr(\fra) = (H, D(A))_{2,\frac 1 2}$~\cite[Proposition~1.13]{Lun09}, the real interpolation space
between $H$ and $D(A)$ which coincides with the complex interpolation space $[H, D(A)]_{\frac 1 2}$~\cite[Corollary~4.37]{Lun09}.
This space, in turn, coincides with $D(A^{1/2})$ because $A$ has bounded imaginary powers
(\cite[Theorem~6.6.9]{Haa06}).
Hence, in the autonomous case for each $f \in L^2(0,T;H)$ and $u_0 \in D(A^{1/2})$ there is a unique $u\in \MRs_\fra(H)$ satisfying~\eqref{eq:ncp}. By McIntosh's example in Section~\ref{sec:autonomous_forms} it may well happen that $V \not\subset D(A^{1/2})$. Then there exists $u_0 \in V$ for which the solution $u \in \MRs(V,V')$ of~\eqref{eq:ncp} is not in $H^1(0,T;H)$.

\section{A first counterexample}\label{sec:counterexample}

For a long time it was not known whether each coercive non-autonomous form has maximal regularity in $H$.
Even though Lions only asked the problem explicitly for symmetric forms (see below),
no counterexample, even to the general case, seemed to be known.
The first counterexample was given by Dier \cite{Die14}. It is based on McIntosh's example
of an autonomous form which fails the square root property.
We reproduce this example, because it is easy and shows the close link between
 the square root property and maximal regularity in $H$.
 
 \begin{example}\cite[Section~5.2]{Die14}\label{ex:counterexample_dier}
 	There exists a non-autonomous, coercive form $\fra$ for which maximal regularity in $H$ fails.
 \end{example}
 \begin{proof}
	Let $\frb \colon V \times V \to \C$ be an autonomous coercive form with associated operator $B$ on $H$ satisfying $D(B^{1/2}) \not\subset V$.
	Such a form exists by the result of McInstoh mentioned in Section~\ref{sec:autonomous_forms}.
	Let  $\frc(v,w) = \tfrac 1 2 (\frb(v,w)+ \overline{\frb(w,v)})$ be the symmetric part of $\frb$
	and $\mathcal C \in \L(V,V')$ the operator associated with $\frc$ and $C$ its part in $H$.
	Then $D(C^{1/2})=V$. Define the form $\fra \colon [0,2]\times V \times V \to \C$ by
	\[
		\fra(t,v,w):= \1_{[0,1)}(t) \frb(v,w) + \1_{[1,2]}(t) \frc(v,w).
	\]
	Then $\fra$ is a non-autonomous, coercive form.
	Let $\A(t) \in \L(V,V')$ be the associated operator and $A(t)$ its part in $H$.
	Then $A(t)=B$ for $t<1$ and $A(t)=C$ for $t\ge1$.
	Let $u_1 \in D(B^{1/2}) \setminus V$.
	Then there exists $\psi \in H^1(0,1;H)\cap L^2(0,1;D(B))$ such that $\psi (0)=u_1$ (since $\Tr(\frb)= D(B^{1/2})$, see Section~\ref{sec:autonomous_forms:regularity}).
	Let $v(t) = t \psi(1-t)$. Then $v \in H^1(0,1;H)\cap L^2(0,1;D(B))$, $v(0)=0$ and $v(1)=u_1$.
	Let $f(t) = (v'(t)+B v(t)) \mathds{1}_{[0,1)}$. Then $f \in L^2(0,2;H)$.
	Let $w \in H^1(1,2;V')\cap L^2(1,2;V)$ be the solution of 
	$w'(t)+ \mathcal C w(t) = 0$, $w(1) = u_1$.
	Then $w \not\in H^1(1,2;H)\cap L^2(1,2;D(C))$ since $u_1 \not\in V = D(C^{1/2}) = \Tr(\frc)$.
	Let $u(t) := v(t) \mathds{1}_{[0,1)} + w(t) \mathds{1}_{[1,2]}$.
	Then $u \in \MRs(V,V')$ is the solution of $u'(t)+\A(t) = f(t)$, $u(0) =0$, but $u \not\in H^1(0,2;H)$.
	Thus the form $\fra$ fails maximal regularity in $H$.
\end{proof}

\section{Symmetric forms}\label{sec:symmtric_forms}

The form in the previous example is not symmetric and continuous. Recall that an autonomous, 
symmetric form does satisfy the square root property, so a construction similar to that in Section~\ref{sec:counterexample}
is not possible.
Indeed, under an additional regularity hypothesis Lions proved the following.

\begin{theorem}[{\cite[IV Sec.~6, Théorème~6.1]{Lio61}}]
	Let $\fra \colon [0,T] \times V \times V \to \K$
	be a non-autonomous form satisfying
	\begin{itemize}
		\item[(a)] $\fra(t,v,w) = \overline{\fra(t,w,v)}$ for all $t \in [0,T]$, $v,w \in V$ \emph{(symmetriy)}
		\item[(b)] $\fra(\cdot,v,w) \in C^1([0,T])$ for all $v,w \in V$.
	\end{itemize}
	Then $\fra$ has maximal regularity in $H$.
\end{theorem}

Lions \cite[p.~68]{Lio61} asks whether this result remains true if the form is merely continuous or even does not satisfy any regularity
besides our general assumption of measurablility.
Fackler recently gave a negative answer to Lions' problem.

\begin{theorem}[\cite{Fac16c}]\label{thm:counterexample_lions}
	There exists a coercive, symmetric, non-autonomous form $\fra$ satisfying
	\[
		\abs{\fra(t,v,w)- \fra(s,v,w)} \le K \abs{t-s}^{1/2} \norm{v}_V \norm{w}_V
	\]
	for all $v,w \in V$, $t,s \in [0,T]$ and some constant $K>0$
	which does not satisfy maximal regularity in $H$.
\end{theorem}
Thus, Lions' problem (exactly as formulated by Lions) has a negative answer even
for a symmetric non-autonomous form which is Hölder continuous in time.
The Hölder index $\frac 1 2$ is the worst possible case as we will see in the next section.

\section{Hölder regularity}

If the form is Hölder continuous of index $\beta>\frac 1 2$, then it has maximal regularity in $H$.
In the following result by Ouhabaz and Spina it is remarkable that the hypothesis of symmetry is no longer needed.
\begin{theorem}[{\cite{OuhSpi10}}]\label{thm:ouhabaz_spina}
	Let $\fra$ be a non-autonomous, coercive form such that
	\[
		\abs{\fra(t,v,w)- \fra(s,v,w)} \le K \abs{t-s}^\beta \norm{v}_V \norm{w}_V
	\]
	for all $v,w \in V$, $t,s \in [0,T]$ and some constants $K>0$ and $\beta > \frac 1 2$.
	Then $\fra$ has maximal regularity in $H$.
\end{theorem}

An $L^p$-version of maximal regularity in $H$ is proved by Haak and Ouhabaz \cite{HaaOuh15}. They also show that $D(A(0)^{1/2})$ is contained in the trace space, thus obtaining a final result in the Hölder scale. The original proof in \cite{OuhSpi10} is based on a result by Hieber--Monniaux \cite{HieMon00b} on non-autonomous evolution equations
satisfying the Aquistapace--Terreni condition. This, in turn, needs a boundedness result for pseudodifferential operators. We refer to~\cite[Theorem~5]{PorStr06} and~\cite[Theorem~17]{HytPor08}, which extend a scalar-valued characterization in~\cite[Theorem~2]{Yam86}.

We now assume that $\fra$ satisfies the square root property uniformly. Then it makes sense to ask whether the embedding $\MRs_{\fra}(H) \hookrightarrow C([0,T];V)$ holds true as in the autonomous case. Indeed, if $\fra$ is a \emph{Lipschitz continuous form}, i.e.\ $\beta = 1$, then a positive answer is contained in~\cite{ADLO14}. More recently, this was generalized to  $\beta > \frac{1}{2}$ by Achache and Ouhabaz~\cite[Theorem~4.2]{AchOuh16}.

\section{Bounded variation}\label{sec:bv}

Another regularity condition, weaker than Lipschitz continuity and not comparable to Hölder continuity, is boundedness of the variation.
A non-autonomous form $\fra$ is of \emph{bounded variation} if
\begin{equation*}
	\sup_{(\tau_k)} \sum_{k=1}^n \norm{\mathcal{A}(\tau_{k}) - \mathcal{A}(\tau_{k-1})} < \infty,
\end{equation*}
where the supremum is taken over all finite partitions $0 = \tau_0 < \tau_1 < \ldots < \tau_{n-1} < \tau_n = T$ of $[0,T]$, or equivalently, there exists $g \colon [0,T] \to \R$ non-decreasing with 
\[
	 \abs{\fra(t,v,w)- \fra(s,v,w)} \le (g(t)-g(s)) \norm{v}_V \norm{w}_V \qquad \text{for all } t,s \in [0,T], v,w \in V.
\]
Let $\fra$ be a coercive, bounded non-autonomous form of bounded variation.

\begin{theorem}[\cite{Die15}]\label{thm:dier_bv}
	If $\fra$ is symmetric, then $\fra$ has maximal regularity in $H$ and $\MRs_{\fra}(H) \hookrightarrow C([0,T];V)$.
\end{theorem}

The inclusion of $\MRs_\fra(H)$ in $C([0,T];V)$ is the main difficulty in the result. We mention that El-Mennaoui and Laasri \cite{MenHaf16} showed that for symmetric forms of bounded variation the solution can be approximated by the solutions of piecewise autonomous approximating problems. The following result shows that in Theorem~\ref{thm:dier_bv} the symmetry condition can be relaxed (keeping the condition of bounded variation).

\begin{theorem}[\cite{Fac16d}]
	If $\fra$ satisfies a parameterized variant of the square root property (see~\cite[Definition~2.3]{Fac16d}), then $\fra$ has maximal regularity in $H$. 
\end{theorem}

The parameterized variant of the square root property is, for example, satisfied for elliptic operators on bounded Lipschitz domains with Neumann or Dirichlet boundary conditions as a consequence of~\cite{AusTch03}.

\section{Fractional Sobolev regularity}

Hölder continuity of order $\beta > \frac 1 2$ and bounded variation are two different non-comparable regularity conditions. 
The following result introduces a new regularity condition on the form which generalizes Hölder continuity and almost contains bounded variation.
Suppose that the operator $\A(\cdot)$ associated to $\fra$ belongs to the \emph{homogeneous fractional Sobolev space} $\mathring W^{1/2+\delta, 2}(I; \mathcal{L}(V,V'))$ for some $\delta>0$; i.e.,
	\[
		\int_I\int_I \frac{\norm{\A(t)-\A(s)}_{\mathcal{L}(V,V')}^2}{\abs{t-s}^{2+2\delta}}  \  \mathrm{d}{t}\ \mathrm{d}{s} < \infty.
	\]
\begin{theorem}[\cite{DieZac16}]\label{thm:DZ}
	Then $\fra$ satisfies maximal regularity in $H$ and $\Tr(\fra) = D(A(0)^{1/2})$.
	 Moreover, $\MRs_\fra(H)$ embeds continuously in $H^{1/2}(I;V)$.
\end{theorem}

The proof is surprisingly elementary and based on the Lax--Milgram lemma. For an $L^p$-version of this result see~\cite[Corollary~5.7]{Fac16e}.

\section{The critical case}\label{sec:critical}

Note that both classes of sufficient regularity conditions on the form, namely Hölder and fractional Sobolev regularity, are special instances of the homogeneous Besov scale defined for a non-autonomous form $\fra$, its associated operator $\mathcal{A}$, an interval $I$ and indices $s \in (0,1)$, $p,q \in [1, \infty]$ via the semi-norm
\begin{equation*}
	\norm{\mathcal{A}}_{\mathring B^{s,p}_q(I)} = \biggl( \int_0^{\infty} \biggl( \frac{1}{h^s } \biggl( \int_{I_h} \norm{\mathcal{A}(t+h) - \mathcal{A}(t)}^p \ \mathrm{d}{t} \biggr)^{1/p} \biggr)^{q} \frac{\mathrm{d}{h}}{h} \biggr)^{1/q},
\end{equation*}
where $I_h = \{ t \in I: t+h \in I \}$ and where one uses the usual modifications for $p = \infty$ or $q = \infty$. Observe that $\mathcal{A}$ is $\beta$-Hölder continuous if and only if $\mathcal{A} \in \mathring B^{\beta,\infty}_{\infty}(I)$. Further, one has $\mathring B^{s,p}_p(I) = \mathring W^{s,p}(I)$ for all $p \in [1, \infty]$ and $s \in (0,1)$. With the positive results and counterexamples discussed in the previous sections we are now able to give a rather complete picture of non-autonomous maximal regularity for forms and to identify a critical case in the Sobolev scale.

On the positive side it follows from embedding results for Besov spaces~\cite{Sim90} that $\mathring{B}^{s,p}_q(I) \hookrightarrow \mathring{W}^{\frac{1}{2} + \delta, 2}(I)$ for some $\delta > 0$ if $s > \frac{1}{2}$ and the \emph{Sobolev index} $s - \frac{1}{p}$ is positive. On the other hand, for $s < 1/2$ the counterexample stated in Theorem~\ref{thm:counterexample_lions} shows, since $C^{\frac{1}{2}}(I) \hookrightarrow \mathring{B}^{s,p}_q(I)$ for all $s < \frac{1}{2}$ and arbitrary $p,q \in [1, \infty]$, that maximal regularity does not hold in the case $s < \frac{1}{2}$ for any $p,q \in [1, \infty]$. Further, if $s \ge \frac{1}{2}$ and the Sobolev index $s - \frac{1}{p}$ is negative, then $\mathring{B}^{s,p}_q(I)$ contains step functions and therefore Example~\ref{ex:counterexample_dier} shows that maximal regularity fails, at least in the absence of the Kato square root property. Hence, what remains open are the cases of $s = \frac{1}{2}$ and non-negative Sobolev index, i.e.\ $p \ge 2$, and of $s > \frac{1}{2}$ and zero Sobolev index, i.e.\ $s = \frac{1}{p}$. 

Note that in the second case one has $\mathring{B}^{\frac{1}{p},p}_q(I) \hookrightarrow \mathring{B}^{\frac{1}{2},2}_q(I)$ and, further, that we know a positive answer for symmetric forms or forms satisfying a parameterized variant of the square root property in the boundary case $s = p = q = 1$, i.e.\ $\mathcal{A} \in \mathring{W}^{1,1}(I)$, because such a form a fortiori has bounded variation and therefore the positive results of Section~\ref{sec:bv} apply. Using~\cite[Example~8.1]{Fac16e} in the Besov scale, we obtain the following new result for the case $s = \frac{1}{2}$.

\begin{proposition}
	Given $p \in [1, \infty]$ and $q \in (2, \infty]$, there exists a coercive, symmetric, non-autonomous form $\fra$ with $\mathcal{A} \in \mathring{B}^{\frac{1}{2},p}_q(I)$ that does not satisfy maximal regularity in $H$.
\end{proposition}
\begin{proof}
	Following the line of arguments in~\cite[Example~8.1]{Fac16e} and~\cite[Section~5]{Fac16c}, we choose $V = L^2(0,\frac{1}{2},w)$ for $w(x) = (x \abs{\log x})^{-3/2}$ and $H = L^2(0,\frac{1}{2})$. Further, we let $u(t,x) = c(x)(\sin(t \varphi(x)) + d)$ for some sufficiently large $d$ and $c(x) = x \abs{\log x}$ as well as $\varphi(x) = w(x)$. One can then show as in~\cite[Section~5]{Fac16c} that $u$ solves the non-autonomous Cauchy problem associated to some non-autonomous symmetric, bounded coercive form $\fra\colon [0,T] \times V \times V \to \R$, some initial value $u_0 \in V$ and some inhomogeneous part $f \in L^2(0,T;V)$. Further, $\mathcal{A}$ belongs to $\mathring{B}^{s,p}_q([0,T];\mathcal{L}(V,V'))$ if $u \in \mathring{B}^{s,p}_q([0,T];V)$. This is what we now verify explicitly. We have for the case $p, q < \infty$
	\begin{align*}
		\MoveEqLeft \norm{u}_{\mathring{B}^{s,p}_q([0,T];V)}^q = \int_0^T \biggl( \int_{I_h} \biggl( \int_0^{\frac{1}{2}} \abs{u(t,x) - u(t+h,x)}^2 w(x) \ \mathrm{d}{x} \biggr)^{\frac{p}{2}} \mathrm{d}{t} \biggr)^{\frac{q}{p}} \frac{\mathrm{d} h}{h^{1+sq}}
	\end{align*}
	As in~\cite[Example~6]{Fac16e} we split the inner integral. For this let $\psi(h) = 2h^{3/2} \abs{\log h}^{3/2}$. If $x \le \psi^{-1}(h)$, we estimate the sinus term trivially and have
	\begin{equation*}
		\abs{u(t,x) - u(t+h,x)}^2 w(x) \le x^{1/2} \abs{\log x}^{1/2}.
	\end{equation*}
	For the innermost term we obtain for $F(x) = x^{3/2} \abs{\log x}^{1/2}$ the upper estimate
		\begin{align*}
			\MoveEqLeft \int_0^{\psi^{-1}(h)} x^{1/2} \abs{\log x}^{1/2} \ \mathrm{d}{x} \lesssim \int_0^{\psi^{-1}(h)} F'(x) \ \mathrm{d}{x} = F(\psi^{-1}(h)) \\
			& \lesssim \psi(\psi^{-1}(h)) \normalabs{\log \psi^{-1}(h)}^{-1} = h \normalabs{\log \psi^{-1}(h)}^{-1} \lesssim h \abs{\log h}^{-1}.
		\end{align*}
	Hence, one part of the triple integral can be estimated up to constants by
	\begin{equation*}
		\begin{split}
			\int_0^T \biggl( \int_{0}^T \biggl( \int_0^{\psi^{-1}(h)} x^{\frac{1}{2}} \abs{\log x}^{\frac{1}{2}} \ \mathrm{d}{x} \biggr)^{\frac{p}{2}} \mathrm{d}{t} \biggr)^{\frac{q}{p}} \frac{\mathrm{d} h}{h^{1+sq}} \lesssim \int_0^T  h^{\frac{q}{2} - sq - 1} \abs{\log h}^{-\frac{q}{2}} \mathrm{d}{h}.
		\end{split}
	\end{equation*}
	If $x \ge \psi(h)$, we use the mean value theorem to obtain the estimate
	\begin{equation*}
		\abs{u(t,x) - u(t+h,x)}^2 w(x) \le h^2 x^{-5/2} \abs{\log x}^{-5/2}.
	\end{equation*}
	Now, the innermost integral is estimated for $F(x) = -x^{-3/2} \abs{\log x}^{-5/2}$ by
		\begin{align*}
			\MoveEqLeft \int_{\psi^{-1}(h)}^{\frac{1}{2}} x^{-5/2} \abs{\log x}^{-5/2} \ \mathrm{d}{x} \lesssim \int_{\psi^{-1}(h)}^{\frac{1}{2}} F'(x) \ \mathrm{d}{x} \le - F(\psi^{-1}(\abs{r})) \\
			& \lesssim \frac{1}{\psi(\psi^{-1}(h))} \normalabs{\log \psi^{-1}(h)}^{-1} \lesssim h^{-1} \abs{\log h}^{-1}.
		\end{align*}
	Consequently, the second part of the integral is dominated up to constants by
	\begin{equation*}
		\int_0^T \biggl( \int_{0}^T \biggl( \int_{\psi^{-1}(h)}^{\frac{1}{2}} h^2 x^{-\frac{5}{2}} \abs{\log x}^{-\frac{5}{2}} \ \mathrm{d}{x} \biggr)^{\frac{p}{2}} \mathrm{d}{t} \biggr)^{\frac{q}{p}} \frac{\mathrm{d} h}{h^{1+sq}} \lesssim \int_0^T h^{\frac{q}{2} - sq - 1} \abs{\log h}^{-\frac{q}{2}} \mathrm{d}{h}.
	\end{equation*}
	Hence, $\norm{u}_{\mathring{B}^{s,p}_q([0,T];V)}^q$ is dominated by the above term on the right hand side. One sees that for $s = \frac{1}{2}$ this integral is finite if and only if $q > 2$. Hence, $\fra$ has the claimed regularity.
\end{proof}

By the above counterexample maximal regularity fails for forms in $\mathring{B}^{\frac{1}{2},2}_q(I)$ for all $q > 2$ and for forms in $\mathring{W}^{\frac{1}{2},p}(I) = \mathring{B}^{\frac{1}{2},p}_p(I)$ for all $p > 2$. Further, by Example~\ref{ex:counterexample_dier} maximal regularity may fail for forms in $\mathring{B}^{\frac{1}{2},p}_{p}(I)$ for all $p < 2$. Together with the positive results this clearly identifies one critical case for maximal regularity of forms which we pose as an open problem.

\begin{problem}
	Let $\fra$ be a non-autonomous, bounded, coercive form with $\mathcal{A} \in \mathring{W}^{\frac{1}{2},2}(I)$. Does maximal regularity in $H$ hold if
	\begin{enumerate}
		\item $\fra$ is symmetric?
		\item $\fra$ satisfies the Kato square root property uniformly?
		\item $\fra$ is arbitrary?
	\end{enumerate}
\end{problem}

As a final conclusion, in Figure~\ref{fig:mr}, we illustrate the validity of maximal regularity of non-autonomous coercive bounded forms under the fractional Sobolev regularity $\mathcal{A} \in \mathring{W}^{s,p}(I)$. If $(\frac{1}{p},s)$ is in the white area, maximal regularity holds for arbitrary forms, whereas maximal regularity may fail in the grey areas. In the dark grey area this is known for symmetric forms, whereas counterexamples in the light grey area are only known for non-autonomous forms violating the Kato square root property. At the ``inner'' boundary between the white and the grey areas maximal regularity is known to fail for the solid line part, whereas the problem of maximal regularity is open for the dashed part of the boundary. In particular, the point in the middle of the square is the most critical. 

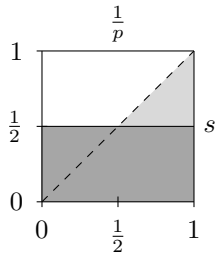
\begin{figure}[ht]\label{fig:mr}
	\centering
    \begin{tikzpicture}
		\fill[gray!30] (0,0) -- (2,2) -- (2,0) -- cycle;
		\fill[gray!70] (0,0) -- (0,1) -- (2,1) -- (2,0) -- cycle;
        \draw[dashed] (0,0) -- (2,2);
        \draw (0,1) -- (2,1);
        \draw (-2pt,0) -- (0,0) node[label={below:$0$}]{} -- (1,0) node[label={below:$\frac 1 2$}]{} -- (2,0) node[label={below:$1$}]{} ;
        \draw (0,-2pt) -- (0,0) node[label = {left:$0$}]{} -- (0,1) node[label = {left:$\frac 1 2$}]{} -- (0,2) node[label = {left:$1$}]{};
        \draw (1,-2pt) -- (1,2pt);
        \draw (-2pt,1)--(2pt,1);
        \draw (2,-2pt)-- (2,1) coordinate[label = {right:$s$}]{} -- (2,2);
        \draw (-2pt,2)--(1,2) coordinate[label = {above:$\frac 1 p$}]{} -- (2,2);
    \end{tikzpicture}
    \caption{Non-Autonomous maximal regularity for $\mathcal{A} \in \mathring{W}^{s,p}(I)$}
\end{figure}

\section{Perturbation results}

As shown by Fackler's example, the Hölder exponent $\frac{1}{2}$ is optimal for maximal regularity in $H$ even in the symmetric case. However, if one improves the boundedness condition on the form, one can allow a weaker Hölder exponent. The following result can be seen as a non-autonomous perturbation of lower order of an autonomous problem. For example, it is well suited for treating non-autonomous Robin boundary conditions.
	
	\begin{theorem}[\cite{AreMon16}]
		Let $0 < \gamma < 1$, $V_{\gamma} = [H,V]_{\gamma}$ and let $\fra$ be a coercive non-autonomous form satisfying
		\begin{equation*}
			\abs{\fra(t,v,w) - \fra(s,v,w)} \le K \abs{t-s}^{\beta} \norm{v}_V \norm{w}_{V_{\gamma}}
		\end{equation*}
		for all $t, s \in [0,T]$, $v, w \in V$, some $K \ge 0$ and some $\beta > \frac{\gamma}{2}$. Then $\fra$ has maximal regularity in $H$. Moreover, the solution space $\MRs_\fra(H)$ is included in $C([0,T];V)$. In particular, $\Tr(\fra) = V$.
	\end{theorem}
	
	This result has been extended to maximal regularity in $L^p$, $1 < p < \infty$, by Ouhabaz~\cite{Ouh15}.
	Moreover, the result stated above even holds for a weaker Sobolev regularity condition in the spirit of Theorem~\ref{thm:DZ} (see \cite{DieZac16}, where the above result is treated as a perturbation result for non-autonomous equations).
	In the limit case $\gamma = 0$, i.e.\
\begin{equation*}
	\abs{\fra_2(t,v,w)} \le K \norm{v}_V \norm{w}_H \qquad \text{for all } v,w \in V,\ t \in [0,T]
\end{equation*}
	no additional time regularity on the perturbation is needed. For additive perturbations this is done in~\cite{ADLO14}, \cite{Die15} and~\cite{DieZac16}. For multiplicative perturbations we refer to~\cite{ADLO14}, \cite{Die15}, \cite{AugJacLas15} and~\cite{AchOuh16}.

The study of perturbations goes back to the classical works of~\cite[VIII Sec.~1]{Lio61} and Bardos~\cite{Bar71}.
In the context of his perturbation results Lions asks again \cite[p.~154]{Lio61} how far regularity assumptions can be reduced.
Now, we have quite precise answers to this question.	

\section{Elliptic operators}

	In an abstract setting, Section~\ref{sec:counterexample} and \ref{sec:symmtric_forms} show that maximal regularity is not valid in general and that the positive results are already close to optimal conditions. However, so far no counterexample seems to be known for elliptic operators. And indeed, it is known that the square root property holds for forms associated with elliptic operators, even in a uniform sense, i.e.\ $D(A^{1/2}) = D(V)$ with equivalent norms, and the corresponding constants only depending on the ellipticity constants. This is exactly what the positive solution of the famous Kato square root problem says~\cite{AHL+02}. For the Kato square root property on Lipschitz domains we refer to~\cite{AusTch03} and for mixed boundary conditions to~\cite{AKM06},~\cite{EgeHalTol16} and~\cite{EgeHalTol14}. Here is a formulation of the problem of non-autonomous maximal regularity for elliptic operators.

	Let $\Omega \subset \R^d$ be open, $a_{ij}\colon [0,T] \times \Omega \to \R$ be bounded and measurable with
	\begin{equation*}
		\sum_{i,j=1}^d a_{ij}(t,x) \xi_i \xi_j \ge \eta \abs{\xi}^2
	\end{equation*}
	for almost all $(t,x) \in [0,T] \times \Omega$ and all $\xi \in \R^d$, where $\eta > 0$. Let $V$ be a closed subspace of $H^1(\Omega)$ containing $H_0^1(\Omega)$. Then
	\begin{equation*}
		\fra(t,v,w) = \int_{\Omega} \sum_{i,j=1}^d a_{ij}(t,x) \partial_i v(x) \overline{\partial_j w(x)} \ \d x
	\end{equation*}
	defines a coercive non-autonomous form on $[0,T] \times V \times V$. Let $H = L^2(\Omega)$.
	
	\begin{problem}\label{prob:mr_elliptic}
		Does this form $\fra$ have maximal regularity in $H$?
	\end{problem}
	
	We expect that further conditions on the coefficients are needed, but they might be weaker than those we encountered for abstract forms. A first result in this direction was obtained very recently by Auscher and Egert.
	
	\begin{theorem}[\cite{AusEge16}]\label{thm:auscher_egert}
		Assume in addition to the assumptions made above that there exists $M \ge 0$ such that
		\begin{equation}
			\label{eq:auscheregert_condition}
			\sup_{I} \frac{1}{\abs{I}} \int_I \int_I \frac{\abs{A(t,x) - A(s,x)}^2}{\abs{t-s}^2} \ \d s \ \d t \le M \qquad (\text{a.e.\ } x \in \Omega).
		\end{equation}
		Then the form $\fra$ satisfies maximal regularity in $H = L^2(\Omega)$. Here $A(t,x) = (a_{ij}(t,x))_{i,j = 1, \ldots, d}$ and the supremum is taken over all intervals $I \subset [0,T]$, $\abs{I}$ denoting the length of $I$.
	\end{theorem}
	
	Condition~\eqref{eq:auscheregert_condition} is stronger than Hölder continuity with index $\frac{1}{2}$, but weaker than Hölder continuity with index $\alpha > \frac{1}{2}$ which is the hypothesis in Theorem~\ref{thm:ouhabaz_spina} by Ouhabaz and Spina. It is also shown by Auscher and Egert that in the situation of Theorem~\ref{thm:auscher_egert} the solution is in $H^{1/2}([0,T];V)$, a new regularity phenomenon first observed in~\cite{DieZac16}. 
	
	Even in dimension $1$ Problem~\ref{prob:mr_elliptic} seems to be open. We want to explain why a positive answer could be of interest for quasilinear parabolic equations. To be as simple as possible, we formulate the following example in dimension $1$.
	
	\begin{example}\label{ex:quasilinear}
		Let $H = L^2(0,1)$, $V = H^1(0,1)$ and let $m\colon (0,1) \to [\delta, \delta^{-1}]$ be continuous, where $0 < \delta < 1$. Let $f \in L^2(0,T;H)$, $u_0 \in H$. Using the compactness of the embedding
		\begin{equation*}
			L^2(0,T;V) \cap H^1(0,T;V') \hookrightarrow L^2(0,T;H) \quad\text{(Aubin--Lions lemma)}
		\end{equation*}
		and Lions' Theorem~\ref{thm:Lions} on maximal regularity in $V'$, one can show with the help of Schauder's fixed point theorem that there exists $u \in \MRs(V,V')$ such that $u(0) = u_0$ and
		\begin{equation*}
			\int_0^1 u'(t) v \ \d x + \int_0^1 m(u(t)) u_x(t) v_x \ \d x = \int_0^1 f(t) v \ \d x \qquad t\text{-a.e.\ }
		\end{equation*}
		for all $v \in H^1(0,1)$, see~\cite[Sec.~4 (II)]{ADLO14} for a proof. Note that the linear part is a non-autonomous problem since we plug in a solution. We would like to prove that $u$ solves the quasilinear parabolic Neumann boundary value problem. If Problem~\ref{prob:mr_elliptic} had a positive answer, then $u \in H^1(0,T;L^2(0,1))$ and then, by Example~\ref{ex:neumann}, we would obtain the following.
		
		Write $u(t,x) := u(t)(x)$. Then for $t \in [0,T]$ a.e.\ we have $u(t,\cdot) \in H^1(0,1)$ and $m(u(t,\cdot))u_x(t, \cdot) \in H^1(0,1)$ and $u(t,x)$ solves the problem
		\begin{equation*}
		 \left\{\begin{aligned}
			u_t(t,x) - (m(u(t,x)) u_x(t,x))_x & = f(t,x)\\
			u(0,x) & = u_0(x),
		\end{aligned} \right.
		\end{equation*}
		as well as $u_x(t,\cdot) \in C[0,1]$ and $u_x(t,1) = u_x(t,0) = 0$ almost everywhere. Thus the solution of the quasilinear problem satisfies Neumann boundary conditions. 
	\end{example}
	
	\begin{remark}
		The problem in one dimension occurs since our operator is in divergence form. If it is in non-divergence form, even in higher dimensions much more can be said (see~\cite{AreChi10}).
	\end{remark}
	
\emergencystretch=0.75em
\printbibliography

\end{document}